\documentclass[a4paper,12pt]{article}

\usepackage{hyperref}

\usepackage{dsfont}
\usepackage{amsfonts}
\usepackage{amsmath}
\usepackage{amssymb}
\usepackage{amsthm}
\usepackage{hyperref}
\usepackage{cleveref}

\newcommand{\setdef}[1]{{\left\{ #1 \right\}}}

\DeclareMathOperator{\Sp}{Sp}

\DeclareMathOperator{\Hom}{Hom}

\DeclareMathOperator{\Spec}{Spec}
\renewcommand{\O}{\operatorname{O}}

\DeclareMathOperator{\Id}{Id}

\def\H{\mathbb{H}}
\def\C{\mathbb{C}}

\def\F{\mathbb{F}}

\def\intersect{\cap}
\def\bigintersect{\bigcap}
\def\isomorphic{\cong}
\def\directsum{\oplus}
\def\bigdirectsum{\bigoplus}
\makeatletter
\newcommand{\tensor}[1][\@nil]{%
  \def\tmp{#1}%
   \ifx\tmp\@nnil
        \mathbin{\otimes}
    \else
         \mathbin{\otimes}_{#1}
    \fi}
\makeatother

\def\comp{\circ}

\def\ideal{\mathbin{\triangleleft}}

\newtheorem{theorem}{Theorem}
\newtheorem{lemma}[theorem]{Lemma}
\newtheorem{corollary}[theorem]{Corollary}
\newtheorem{proposition}[theorem]{Proposition}

\theoremstyle{remark}
\newtheorem{definition}[theorem]{Definition}
\newtheorem{remark}[theorem]{Remark}

\makeatletter

\let\orig@hyper@refstepcounter\hyper@refstepcounter%

\newcommand{\thmcountername}{theorem}

\renewcommand{\hyper@refstepcounter}[1]{%
  \begingroup
  \def\temp@a{#1}%
  \edef\temp@b{\thmcountername}
  \ifx\temp@a\temp@b
  \ifx\@currenvir\temp@b
  \orig@hyper@refstepcounter{#1}%
  \else
  \expandafter\let\csname c@\@currenvir\expandafter\endcsname\csname c@\thmcountername\endcsname
  \orig@hyper@refstepcounter{\@currenvir}%
  \fi
  \else
  \orig@hyper@refstepcounter{#1}%
  \fi
  \endgroup
}

\newcommand{\theolabel}[1]{%
  \label[\@currenvir]{#1}%
}

\makeatother

\usepackage[backend=bibtex]{biblatex}
\usepackage{mathrsfs}
\usepackage{stmaryrd}
\usepackage{mathtools}
\usepackage{thmtools}
\usepackage{thm-restate}
\usepackage{todonotes}
\usepackage{ifthen}
\usepackage{setspace}

\usepackage{BOONDOX-cal}

\usepackage{tikz}
\usepackage{tikz-cd}
\usepackage{tqft}
\usetikzlibrary{decorations.pathreplacing,fit}

\tikzstyle{halftqft}=[tqft/cobordism height=1cm,tqft/circle width=5pt,tqft/circle depth=2.5pt,tqft/boundary separation=1cm]
\tikzstyle{minitqft}=[tqft/cobordism height=1.2em,tqft/circle width=5pt,tqft/boundary separation=17.5pt,tqft/circle depth=1pt]
\tikzstyle{tinytqft}=[tqft/cobordism height=0.6em,tqft/circle width=2.5pt,tqft/boundary separation=8.75pt,tqft/circle depth=0.5pt]

\tikzset{%
    tqft/cobordism height=3em,
    tqft/circle width=5pt,
    tqft/boundary separation=25pt
}

\title{Extended TQFTs and Algebraic Geometry}
\author{Peter Banks}

\DeclareMathOperator{\Ext}{Ext}

\DeclareMathOperator{\Ev}{Ev}

\DeclareMathOperator{\Bord}{Bord}
\DeclareMathOperator{\jacobsonradical}{Jac}
\DeclareMathOperator{\Specm}{Specm}

\renewcommand{\O}{\mathcal{O}}
\newcommand{\Od}{\mathcal{O}_\Delta}
\newcommand{\Odd}{\mathcal{O}_{\Delta_3}}
\newcommand{\Ox}{\mathcal{O}_X}
\newcommand{\Oy}{\mathcal{O}_Y}

\newcommand{\Oxx}{\mathcal{O}_{X \times X}}

\renewcommand{\c}{\mathscr{C}}
\renewcommand{\d}{\mathscr{D}}
\newcommand{\e}{\mathscr{E}}
\newcommand{\f}{\mathscr{F}}
\newcommand{\g}{\mathscr{G}}

\newcommand{\p}{\mathscr{P}}
\newcommand{\q}{\mathscr{Q}}
\renewcommand{\r}{\mathscr{R}}
\newcommand{\x}{\mathscr{X}}

\renewcommand{\C}{\mathcal{C}}
\newcommand{\D}{\mathcal{D}}
\renewcommand{\F}{\mathcal{F}}

\newcommand{\Sh}{\mathcal{SH}}
\newcommand{\catname}[1]{{\normalfont\textbf{#1}}}
\newcommand{\Var}{{\mathcal{V\mkern-6mu a\mkern-2mu r}}}
\newcommand{\AffVar}{{\mathcal{A\mkern-4mu f\mkern-4mu f \mkern-2mu V\mkern-6mu a\mkern-2mu r}}}
\newcommand{\Rmod}{R\catname{-mod}}
\newcommand{\Smod}{S\catname{-mod}}
\newcommand{\Alg}{{\mathcal{A \mkern-3mu l \mkern-3mu g}}}
\newcommand{\ABBimod}{A\catname{-}B\catname{-bimod}}
\newcommand{\ABmod}{(A \tensor_k B)\catname{-mod}}
\newcommand{\DABBimod}{\D(\ABmod)}
\newcommand{\DAlg}{{\Alg_d}}
\renewcommand{\H}{\catname{H}}
\newcommand{\HC}{\catname{H}\C}
\newcommand{\HBord}{\catname{H}\Bord}
\newcommand{\HVar}{\catname{H}\Var}

\newcommand{\AlgProjFn}{\Psi}
\newcommand{\AlgInclFn}{\Phi}
\newcommand{\sheaffunctor}{{\widetilde{\hphantom{M}}}}

\newcommand{\pt}{\setdef{\text{pt}}}
\newcommand{\boxtensor}{\boxtimes}
\newcommand{\m}{\mathfrak{m}}
\newcommand{\n}{\mathfrak{n}}

\newcommand{\bordism}[2][tqft]{%
\begin{tikzpicture}[#1,baseline=(baseline)]%
	\node[tqft/#2,draw] (obj0) {};%
\ifthenelse{\equal{#2}{cup}}
{\coordinate (baseline) at ([yshift=0.2*\pgfkeysvalueof{/pgf/tqft/cobordism height}]obj0.west);}
{
\ifthenelse{\equal{#2}{cap}}
{\coordinate (baseline) at (obj0.outgoing boundary 1);}
{\coordinate (baseline) at (obj0.west);}
}
\end{tikzpicture}%
}

\newcommand{\bordismtensor}[3][tqft]{%
\begin{tikzpicture}[#1,baseline=(baseline)]%
	\node[tqft/#2,draw] (obj0) {};%
	\node[tqft/#3,draw,right=1em of obj0] (obj1) {};%
\ifthenelse{\equal{#2}{cup}}
{\coordinate (baseline) at ([yshift=0.2*\pgfkeysvalueof{/pgf/tqft/cobordism height}]obj0.west);}
{
\ifthenelse{\equal{#2}{cap}}
{\coordinate (baseline) at (obj0.outgoing boundary 1);}
{\coordinate (baseline) at (obj0.west);}
}
\end{tikzpicture}%
}

\newcommand{\bordismtwo}[3][tqft]{%
\begin{tikzpicture}[#1,baseline={([yshift=-0.3em]obj0.outgoing boundary 1)}]%
	\node [tqft/#2,draw] (obj0) {}; %
	\node [tqft/#3,draw,anchor=incoming boundary 1] (obj1) at (obj0.outgoing boundary 1) {};
\end{tikzpicture}
}

\newcommand{\bordismbox}[1][tqft]{%
\begin{tikzpicture}[#1,baseline=(baseline)]
	\node [tqft/cap] (cap0) {};
	\node [tqft/cup,anchor=incoming boundary 1] (cup0) at (cap0.outgoing boundary 1) {};
	\node [coordinate] (topleft)  at ([xshift=-6pt]cap0.north) {};
	\node [coordinate] (botright) at ([xshift=6pt]cup0.south) {};
	\node [fit=(topleft)(botright),dotted,draw,rectangle] (box) {};
	\coordinate (baseline) at ([yshift=-0.4em]cap0.outgoing boundary 1);
\end{tikzpicture}
}
\newcommand{\bordismcapcup}{%
\begin{tikzpicture}[minitqft,tqft/cobordism height=0.6em,baseline=(baselinecoord)]%
	\node[tqft/cup,draw] (cup) {};%
	\node[tqft/cap,anchor=incoming boundary 1,draw] (cap) at (cup.outgoing boundary 1) {};%
	\coordinate (baselinecoord) at ([yshift=-3pt] cap.incoming boundary 1);
\end{tikzpicture}%
}

\newcommand{\bordismcopants}{%
\begin{tikzpicture}[minitqft,baseline=(xxx)]%
	\node[tqft/reverse pair of pants,draw] (cpants0) {};
	\coordinate (xxx) at ([yshift=-3pt]cpants0.west);
\end{tikzpicture}
}

\begin{document}
\maketitle
\begin{abstract}
We study a potential method for constructing the Rozansky--Witten TQFT as an extended $(1+1+1)$-TQFT. We construct a $2$-category consisting of schemes, complexes of sheaves and sheaf morphisms and show that there are $(1+1)$-TQFTs valued in the truncation of this category which have state spaces that agree with the Rozansky--Witten TQFT. However, we also show that if such a TQFT is based on a reduced Noetherian scheme, it cannot be extended upwards to a $(1+1+1)$-TQFT.
\end{abstract}

\section{Introduction}

After their formalisation by Atiyah\cite{Atiyah1988}, TQFTs have proven to be a rich area of study in mathematics. They have been used to categorify previously known invariants, such as the Jones polynomial, and also give new $3$-manifold invariants. Part of their power comes from allowing intuition about physical systems to be applied to mathematics. Indeed, many TQFTs are defined using path integrals, a notion from physics whereby an invariant of the manifold $M$ is defined by performing an integral over the space of all fields $M \to X$ for some target space $X$.

This was the method used by Rozansky and Witten in defining their TQFT~\cite{RozanskyWitten1997invariants}. They in fact construct a family of TQFTs, parameterised by a hyperk{\"a}hler manifold $X$ (although this has since been generalised by Kontsevich~\cite{Kontsevich1999} and Kapranov~\cite{Kapranov1999} to any sympletic holomorphic manifold). The TQFT itself is a twisted $N=6$ sigma model.

Analysing this TQFT can be used to generate results in different areas of mathematics depending on one's viewpoint. Since the TQFT is parameterised by hyperk{\"a}hler manifolds, it can be used to give insight into geometric problems. It can also be approached from a topological point of view, insofar as it gives a $3$-manifold invariant. Finally, it can also be considered from the original physical viewpoint. 

Since the TQFT is constructed based on a path integral, it is only well-defined at a physical level of rigour. Despite this, many results have been proven: for example, it has been used to derive a formula relating the $\hat{A}$-genus of a hyperk{\"a}hler manifold (a topological property) to its curvature (a geometric property).

In their original paper, Rozansky and Witten give a description of the state spaces of the genus $0$- and $1$-surfaces, and conjecture a general formula. Let $\Sigma_g$ denote the genus $g$ surface. Their conjecture is the formula:
\[ Z(\Sigma_g) = \bigdirectsum_{q=0}^{\dim_\mathbb{C}(X)} H^q_{\overline{\partial}}\left(X, (\wedge T^{1,0} X)^{\tensor g}\right) \,. \]
For any $(n+1)$-TQFT, the space $A = Z(S^n)$ is given the structure of a finite-dimensional commutative Frobenius algebra; we say that $Z$ is \emph{based} on $A$. The results proved in their paper show that the Rozansky--Witten TQFT is based on the algebra $k[\alpha]/\alpha^2$. In itself, this is of some interest. The TQFTs constructed by the Tureav--Viro method must be decomposable into a direct sum of TQFTs based on $k$, so this would give an example of a TQFT not constructed from that method. Furthermore, by the classification result of Bartlett, Douglas, Schommer-Pries and Vicary \cite{BDSV2015}, such a TQFT would be an example of a TQFT that could not be extended to produce a $(1+1+1)$-TQFT valued in $\catname{2Vect}$.

There have been many attempts to give a rigorous construction of this TQFT. The work of Kapustin, Rozansky and Saulina \cite{KRS2009} approaches the problem by attempting to construct the category of boundary conditions of the TQFT. There have also been attempts to formalise the TQFT within the BV-formalism \cite{KCQ2017}.

Instead, the focus of this paper is the conjectured construction given by Roberts and Willerton \cite{RobertsWillerton2003}. They conjecture that the Rozansky--Witten TQFT can be viewed as an extended $(1+1+1)$-TQFT valued in a $2$-category with objects given by derived categories of sheaves, $1$-morphisms functors between these categories and $2$-morphisms natural transformations. In this TQFT, the pair-of-pants bordism is sent to the functor given by taking the tensor product with the diagonal sheaf and taking the inverse image. 

Rather than use this category, we chose our target category to be $\Var$, which has objects given by schemes and $\Hom$-categories given by derived categories of sheaves of modules. A sheaf of modules $\e \in \Sh(X \times Y)$ can be viewed as a functor via the associated Fourier--Mukai transform
\begin{align*}
	\Phi_\e\colon \Sh(X) &\to \Sh(Y) \,, \\
		\f &\mapsto \pi_{Y*}(\pi^*_{X}(\f) \tensor \e) \,.
\end{align*}
The composition law of $\Var$ is chosen to respect this interpretation as a functor: that is, so that $\Phi_{\f \comp \e} \isomorphic \Phi_\f \comp \Phi_\e$. This situation is analogous to the category of algebras, bimodules and morphisms; in fact, we will see this category is equivalent to the full subcategory of affine schemes of $\Var$.

Using this interpretation of sheaves as functors, a subcategory of the original category considered by Roberts and Willerton can be recovered. In particular, the tensor product and inverse image functors can be represented as Fourier--Mukai transforms, so we can construct the TQFTs that they describe.

The benefits of using this category is that it allows us to define a symmetric monoidal structure. We take the tensor product of two sheaves $X \tensor Y$ to be the product $X \times Y$; we then take the product of two schemes to be their external product (and likewise for $2$-morphisms). We will show in \autoref{sec:VarConstruction} that this is indeed a symmetric monoidal $2$-category.

Having shown this result, we will then show that for any scheme $X$, there are  Frobenius algebra objects in the truncated category $\HVar$ as described by Roberts and Willerton (\autoref{prop:diagonaltqft}). These have unit and counit morphisms given by the structure sheaf $\Ox$; the multiplication and comultiplication are given by the triagonal sheaf $i^1_{123*}(\Ox)$, where $i^1_{123}\colon X \to X^3$ is the unique map determined by the property $\pi^{123}_{i} \comp i^1_{123} = \Id_X$, for $i = 1, 2, 3$. We denote the associated $(1+1)$-TQFT by $Z_X$. In certain cases, we can calculate that the state space associated to the genus $g$ surface by this TQFT is isomorphic to that conjectured by Rozansky and Witten (\autoref{prop:PropertyDStateSpaceCalculation}).

However, these TQFTs cannot, in general, be extended upwards to $(1+1+1)$-TQFTs. This is the main result of this paper.

\begin{theorem}
\label{thm:ETQFTDiscrete}
Let $Z$ be a $(1+1+1)$-TQFT valued in $\Var$ such that the induced $(1+1)$-TQFT corresponds is of the form $Z_X$ defined above, where $X = Z(S^1)$. If $X$ is of finite type and reduced, then it must be discrete. In this case, $Z$ is isomorphic to a direct sum of extended TQFTs, each of which sends $S^1$ to a single point.
\end{theorem}

Consequently, a $(1+1+1)$-TQFT valued in $\Var$ and based on a Noetherian reduced scheme must either truncate to a different Frobenius algebra object, or be trivial. 

Throughout this paper, we fix a field $k$; all schemes will be schemes over this field. The product of two schemes $X \times Y$ will mean the fibre product over $\Spec(k)$.

I would like to thank my supervisor Andr{\'a}s Juh{\'a}sz for his support throughout the development of this paper. I would also like to thank Andr{\'e} Henriques and Bal{\'a}zs Szendr{\"o}i for their advice and comments on the paper, and Nils Carqueville for discussions regarding the construction of the Rozansky-Witten TQFT. This project has received funding from the European Research Council (ERC) under the European Union's Horizon 2020 research and innovation programme (grant agreement No 674978).

\newcommand{\Shapes}{\catname{Sh}}
\newcommand{\Spaces}{\catname{Sp}}
\newcommand{\Darr}{\catname{D}_1}
\newcommand{\Dobj}{\catname{D}_0}
\newcommand{\fmcomp}{\odot}

\newcommand{\tikzcdsquare}[9][row sep=large,column sep=large]{
\begin{tikzcd}[#1]
	#2 \arrow[r,"{#8}"{name=E}] \arrow[d,"{#6}"] \pgfmatrixnextcell #3 \arrow[d,"{#7}"] \\
	#4 \arrow[r,"{#9}"{name=F}] \pgfmatrixnextcell #5
\end{tikzcd}
}

\newcommand{\dblver}[4][column sep=large,row sep=large]{%
\begin{tikzcd}[#1]
	#2 \arrow[d,"{#4}"] \\
	#3
\end{tikzcd}
}
\newcommand{\dblhor}[4][column sep=large,row sep=large]{
\begin{tikzcd}[#1]
	#2 \arrow[r,"{#4}"] \pgfmatrixnextcell #3
\end{tikzcd}
}
\newcommand{\dblcell}[2][column sep=large,row sep=large]{%
	\def\tempa{#2}
	\dblcellreal{#1}
}
\newcommand{\dblcellreal}[9]{
\begin{tikzcd}[#1]
	\tempa \arrow[r,"{#7}"{name=E}] \arrow[d,"{#5}"] \pgfmatrixnextcell #2 \arrow[d,"{#6}"] \\
	#3 \arrow[r,"{#8}"{name=F}] \pgfmatrixnextcell #4
\arrow[Rightarrow,from=F,to=E,"{#9}"]
\end{tikzcd}
}

\section{\texorpdfstring{$2$}{2}-category of schemes}
\label{sec:VarConstruction}

We briefly recall the description of the category $\Var$ that will serve as the target category for the TQFTs we construct. A full description of its construction, along with an introduction to the notions of $2$-categories we will use, can be found in \cite{CaldararuWillerton2010}.

The objects of $\Var$ are smooth schemes over $k$. Its hom-categories are given by
\[ \Hom_\Var(X, Y) = \D(X \times Y) \,, \]
where $\D(X \times Y)$ is the derived category of quasi-coherent sheaves of modules on $X \times Y$. Consequently, the $2$-morphisms in $\Var$ are $1$-morphisms in these derived categories, so vertical composition of these $2$-morphisms to be the usual composition in the derived category. The horizontal composition of two sheaves is defined by their composition as Fourier--Mukai kernels,
\[
	\e \fmcomp \f = \pi^{ABC}_{AC*}(\pi^{ABC*}_{AB}(\e) \tensor \pi^{ABC*}_{BC}(\f))
\]

This category can be viewed as a generalisation of the category $\Alg$, where objects are rings, $1$-morphism are bimodules and $2$-morphisms are morphisms of bimodules; indeed, this category is isomorphic to the full subcategory of affine schemes (see \autoref{sec:AffineSubcategory} for more details). 

In the following sections, we will show how this category can be endowed with the structure of a symmetric monoidal category, a necessary requirement for it to serve as the target for a TQFT. We construct a symmetric monoidal double category whose associated horizontal category is isomorphic to $\Var$; the monoidal structure of the double category then gives a monoidal structure on $\Var$.

This construction starts by considering a geometrically fibered (henceforth ``geofibered'') category \cite{Reich14}. A geofibered category includes the data of functor
\[
	F\colon \Shapes \to \Spaces
\]
from a category of shapes to a category of spaces; for any $f \in \Hom_\Spaces(A, B)$, it also gives a pair of functors
\[
	f_*\colon \Shapes_A \leftrightarrows \Shapes_B\colon f^* \,.
\]
Here, the notation $\Shapes_X$ indicates the fibre category of $\Shapes$ over the object $X$. The main focus of this paper is the case where $\Spaces$ is the category of smooth schemes and morphisms; $\Shapes$ is the category whose objects are tuples $(X, Y, \e)$, where $X, Y \in \Spaces$ and $\e \in \D(X \times Y)$; the functor $F$ sends $(X, Y, \e)$ to $X \times Y$; and $f_*$ and $f^*$ are the direct and inverse image functors respectively. However, the construction can equally be applied to other geofibered category. These include cases where morphisms are morphisms in the non-derived category of sheaves, and some types of stacks.

To give the double category its monoidal structure, it is necessary to introduce a notion of a monoidal structure for a geofibered category (see \autoref{def:MonoidalGeofiberedCat}). This structure is modelled on the external tensor product of sheaves; recall this is defined by the functor
\begin{align*}
	\boxtensor\colon \D(X) \times \D(Y) &\to \D(X \times Y) \\
		(\e, \f) &\mapsto \pi^{XY*}_X(\e) \tensor \pi^{XY*}_Y(\f) \,.
\end{align*}

\subsection{Geofibered categories}

We recall the following definitions from \cite{Reich14}.

\begin{definition}Let $f: X \to Y$ be a morphism of spaces. We call the functors $f_*$ and $f^*$ \emph{basic standard geometric functors}.

We define a \emph{standard geometric functor} to be to be any composition of basic standard geometric functors.
\end{definition}

Consider a Cartesian diagram
\[
	\tikzcdsquare{A}{B}{U}{V}{g'}{g}{f'}{f}
\]
There is a morphism $\text{bc}(f, g)\colon g^* \comp f_* \to f'_* \comp f'^*$, called the \emph{base-change morphism}, formed by applying adjunction transformations to the corner-swapping map
\[
f_* \comp (g')_* \isomorphic (f \comp g')_* = (g \comp f')_* \isomorphic g_* \comp f'_* \,.
\]
A class of morphisms $P$ is \emph{pull-geolocalizing} if:
\begin{itemize}
	\item it contains every isomorphism and is closed under composition; and
	\item for each $f \in \Sp, g \in P$, $\text{bc}(f, g)$ is an isomorphism and $g' \in P$.
\end{itemize}

An example of a pull-geolocalizing class is the collection of flat maps in the geofibered category of schemes~\cite[Proposition 3.9.5]{Lipman09}.

\subsection{Notation}

We will often represent standard geometric functors by directed graphs which are topologically linear. Each vertex is labelled by a space $X$, and each edge is labelled by a morphism $f$, from the space labelling the source vertex to the space labelling the target vertex. After choosing one of the leaves as the source, this determines a functor by taking direct image functors (when travelling in the direction of the edge) and inverse image functors (when travelling against the direction of the edge). For example, the diagram
\[ \begin{tikzcd} X & Y \arrow[l,"g"] \arrow[r,"f"] & Z \end{tikzcd} \]
represents the functor $f_* \comp g^*\colon \Sh_X \to \Sh_Z$ (when read left-to-right).

These diagrams can be stacked vertically, with labelled arrows between each layer, to represent natural transformations and compositions thereof. For example, the base-change morphism corresponding to the pullback diagram
\[
\begin{tikzcd}
    U \arrow[r, "f'"] \arrow[d, "g'"] & X \arrow[d, "g"] \\
    Y \arrow[r, "f"] & Z
\end{tikzcd}
\]
can be depicted as
\[
\begin{tikzcd}
Y \arrow[r,"f"] & Z\arrow[d, Rightarrow,"\text{bc}",red] & X \arrow[l,"g"] \\
Y & U \arrow[l,"g'"] \arrow[r,"f'"] & X
\end{tikzcd}
\]

We will assume that the category of spaces $\Spaces$ has a terminal object, denoted $\pt$ (since in the geofibered category of sheaves over smooth schemes, the terminal object $\pt = \Spec(k)$ is the scheme with a single point). The product of two spaces is then the fibre product over this space. For brevity, when considering products of spaces we will often drop the product symbol; thus the space $X \times Y$ may be denoted by $XY$. Similarly, when no ambiguity may arise, we may write $fg$ for the morphism $f \times g$; composition of morphisms will be explicitly denoted using $\comp$.

In the remainder of the paper, it will frequently be necessary to refer to projection and inclusion morphisms. We define a two families of such morphisms.

\begin{definition}
\label{def:ProjectionMorphisms}
Let $\setdef{X_i : i \in I}$ be a collection of spaces, and let $X_I = \prod X_i$. Define a family of projection morphisms, parameterised by subsets $J \subset K \subset I$, to be the projection $\pi^K_J\colon X_K \to X_J$ onto the factors denoted by the sets of indices $J$.
\end{definition}

\begin{definition}
\label{def:InclusionMorphisms}
Let $\setdef{X_i : i \in I}$ be a collection of spaces. Let $J \subset I$ and let $\setdef{J_j : j \in J}$ be a partition of $I$ such that if $j$ and $k$ are indices in the same partition, then $X_k = X_j$. Define a morphism
\[
i^J_\setdef{J_j}\colon X_J \to X_I
\]
by the property that for any $j \in J$ and $k \in J_j$, we have
\[
	\pi^I_k \comp i^J_\setdef{J_j} = \Id_{X_j} \,.
\]
For example, the morphism $i^1_{12}$ is the diagonal morphism.
\end{definition}

When using diagrams to represent functors, an unlabelled arrow may be used to represent the composition of a projection morphism and an inclusion morphism. This will only be used when the domain of such an arrow has factors that appear uniquely. The projection morphism projects onto any factor that appears in both the domain and codomain. The codomain is partitioned by grouping distinct spaces, and this partition is used to construct the inclusion morphism. For example, the diagram
\[
\begin{tikzcd} XYYZ & XYZ \arrow[l] \arrow[r] & XZ \end{tikzcd} \,,
\]
represents the composition $\pi^{124}_{14*} \comp i^{124*}_{1,23,4}$, where $X_1 = X, X_2 = X_3 = Y, X_4 = Z$.

For any space $X$, there is a unique map $\pi\colon X \to \pt$. We fix a sheaf $\O_\pt$ and define the structure sheaf over $X$ as $\Ox = \pi^*(\O_\pt)$. As an immediate consequence, there are isomorphism $\Ox \isomorphic f^*(\Oy)$ for any $f\colon X \to Y$, using the composition natural transformation.

Finally, we will sometimes use $\O_{\Delta_X}$ to denote the diagonal sheaf $\O_{\Delta_X} = i^1_{12*}(\Ox)$. If the space in question is clear, we may write this as $\Od$.

\subsection{Monoidal geofibered categories}

\begin{definition}
\label{def:MonoidalGeofiberedCat}
A monoidal structure on a geofibered category $F\colon \Shapes \to \Spaces$ is a functor
\[
	\boxtensor \colon \Shapes_X \times \Shapes_Y \to \Shapes_{XY}
\]
along with, for any morphisms $f$ and $g$ in $\Spaces$, the following natural transformations:
\begin{itemize}
	\item an associator $\alpha_\boxtensor: \boxtensor \comp (\Id \times \boxtensor) \to \boxtensor \comp (\boxtensor \times \Id)$;
	\item a left unitor transformation $i^{1*}_{12}(\e \boxtensor \Ox) \isomorphic \e$;
	\item a right unitor transformation $i^{1*}_{12}(\Ox \boxtensor \e) \isomorphic \e$;
	\item a natural transformation $\boxtensor \comp (f_* \times g_*) \isomorphic (f \times g)_* \comp \boxtensor$; and
	\item a natural transformation $\boxtensor \comp (f^* \times g^*) \isomorphic (f \times g)^* \comp \boxtensor$.
\end{itemize}
We require these to satisfy the usual unit and pentagon diagrams for a monoidal product, along with the additional requirement of compatible in the sense that the diagram
\[
\begin{tikzcd}
	\boxtensor \comp (\boxtensor \times \Id) \comp (f_* \times g_* \times h_*) \arrow[r] \arrow [d] &
		\boxtensor \comp (\Id \times \boxtensor) \comp (f_* \times g_* \times h_*) \arrow[d] \\
	\boxtensor \comp ((f \times g)_* \times h_*) (\boxtensor \times \Id) \arrow[d] &
		 \boxtensor \comp (f \times (g_* \times h_*)) (\Id \times \boxtensor) \arrow[d] \\
	(f \times g \times h)_* \comp \boxtensor \comp (\boxtensor \times \Id) \arrow[r] &
		(f \times g \times h)_* \comp \boxtensor \comp (\Id \times \boxtensor)
\end{tikzcd}
\]
commutes; similarly, we require the diagram formed from this by replacing all direct images by inverse images also commutes.
\end{definition}

\begin{remark}
This construction is a generalisation of the external tensor product in the category of sheaves. In this category, we can recover the usual tensor product from this data by the formula
\[	
	\e \tensor \f \isomorphic i^{1*}_{12}(\e \boxtensor \f)
\]
\end{remark}

In any monoidal geofibered category, we can define a notion of Fourier--Mukai transforms. Recall in the traditional setting, the Fourier--Mukai transform associated to a sheaf $\e \in \D(X \times Y)$ is the functor
\begin{align*}
	\Phi_\e\colon \D(X) &\to \D(Y) \\
		\x &\mapsto \pi^{XY}_{Y*}\left( \pi^{XY*}_{X}(\x) \tensor \e \right) \,.
\end{align*}
Since the tensor product can be expressed using the diagonal morphism $i^1_{12}$ and the external product, we can construct an equivalent functor for a monoidal geofibered category:
\begin{align}
	\Phi_\e\colon \Shapes_X &\to \Shapes_Y \nonumber \\
		\x &\mapsto \pi^{XY}_{Y*} \comp i^{1*}_{12} \comp (\pi^{XY}_{X} \times \Id_{XY})^*(\x \boxtensor \e) \label{eqn:FourierMukaiTransformDefinition} \,.
\end{align}

\begin{remark}
\label{rem:kernelsnotunique}
Note that in general, the functor does not determine the kernel \cite{CanonacoStellari2010}. However, there are certain cases where the kernel is determined. If $\Phi_\e\colon \Shapes_{\pt} \to \Shapes_{\pt}$ is a Fourier--Mukai transform, then  all of the morphisms in \autoref{eqn:FourierMukaiTransformDefinition} are isomorphisms, so the kernel can be recovered.
\end{remark}

We can generalise the projection formula for sheaves to monoidal geofibered categories (this is an immediate consequence of the proof given by Reich~\cite{Reich14}).

\begin{lemma}
\label{lem:ProjectionFormulaExterior}
Let $F\colon \Shapes \to \Spaces$ be a monoidal geofibered category. Let $f \in \Hom_\Spaces(X, Y)$. Then there is a natural isomorphism of functors
\[
	f_* \comp \Delta^* \comp (f \times \Id)^* \isomorphic \Delta^* \comp (\Id \times f)_*
\]
\end{lemma}

\begin{corollary}
\label{lem:ProjectionFormulaSpecialised}
Let $\e \in \Shapes_Y$. There is an isomorphism
\[
	f_* \comp f^* (\e) \isomorphic \Delta^* \comp (\e \boxtensor f_*(\Ox))
\]
\end{corollary}

\begin{proof}
Applying \autoref{lem:ProjectionFormulaExterior} to $\e \boxtensor \Ox$ gives
\begin{align*}
	f_* \comp \Delta^* \comp (f \times \Id)^*(\e \boxtensor \Ox)
		&\isomorphic \Delta^* \comp (\Id \times f)_* (\e \boxtensor \Ox) \\
		&\isomorphic \Delta^* \comp (\e \boxtensor f_*(\Ox)) \,.
\end{align*}
The left-hand side is isomorphic to
\[
	f_* \comp \Delta^*(f^*(\e) \boxtensor \Ox) \isomorphic f_* \comp f^*(\e) \,,
\]
where we use the unitor isomorphism $\Delta^* (\e \boxtensor \Ox) \isomorphic \e$. 
\end{proof}

\subsection{Construction of double category}

\begin{proposition}
\label{prop:ConstructingDoubleCategory}
Let $F\colon \Shapes \to \Spaces$ be a monoidal geofibered category, where $\Spaces$ has a terminal object, and the class of projection morphisms is pull-geolocalizing. Then there is a double category with:
\begin{itemize}
	\item objects given by objects in $\Spaces$;
	\item vertical morphisms given by diagrams of the form $\dblver{A}{U}{f}$, where $f \in \Hom_{\Spaces}(A, U)$;
	\item horizontal morphisms given by diagrams $\dblhor{A}{B}{\e}$, where $\e \in \Shapes_{A \times B}$; and
	\item $2$-cells given by diagrams
\[
	\dblcell{A}{B}{U}{V}{f}{g}{\e}{\f}{\alpha}
\]
where:
	\begin{itemize}
		\item $A, B, U, V \in \Spaces$; 
		\item $f \in \Hom(A, U)$, $g \in \Hom(B, V)$;
		\item $\e \in \Shapes_{A \times B}$, $\f \in \Shapes_{U \times V}$; and
		\item $\alpha \in \Hom(\f, (f \times g)_*(\e))$.
	\end{itemize}
\end{itemize}
The double category is such that in its horizontal category, composition of $1$- and $2$-morphisms are given by the Fourier--Mukai composition,
\[
	\e \fmcomp \f = \pi^{ABC}_{AC} \comp i^{124*}_{1234}(\e \boxtensor \f) \,.
\]

Furthermore, this double category can be given the structure of a symmetric monoidal category, with monoidal product given on objects by
\[
	X \tensor Y = X \times Y
\]
and on horizontal morphisms by
\[ \e \tensor \f = \e \boxtensor \f \,. \]
The monoidal product also satisfies that the product of globular $2$-cells is given by
\[
	\dblcell{A}{B}{A}{B}{\Id_A}{\Id_B}{\e}{\p}{\alpha}
		\tensor \dblcell{C}{D}{C}{D}{\Id_C}{\Id_D}{\f}{\q}{\beta} =
	\dblcell{A \times C}{B \times D}{A \times C}{B \times D}{\Id_{A \times C}}{\Id_{B \times D}}{\e \boxtensor \f}{\p \boxtensor \q}{\alpha \boxtensor \beta} \,.
\]
\end{proposition}

\begin{proof}[Proof (sketch)]
We give the details of the structure functors for the double category, and some details of the natural transformations describing their coherences (e.g. associativity). For brevity, we do not give any details of the calculations required to verify that the necessary conditions are  met (e.g. Mac Lane's pentagon diagram), as verifying these conditions offers little insight but are cumbersome to check.

Let $\Dobj = \Spaces$ and $\Darr$ be the category with objects given by diagrams $\dblhor{A}{B}{\e}$ and morphisms given by diagrams 
\[
	\dblcell{A}{B}{U}{V}{f}{g}{\e}{\f}{\alpha}
\]
where the top and bottom lines are objects of $\Darr$, the morphisms $f$ and $g$ are morphisms in $\Spaces$, and
\[ \alpha\colon \F \to (f \times g)_*(\e) \]
is a morphism in $\Shapes_{U \times V}$. Composition of two diagrams (represented by vertically stacking cells) is defined by:
\[
\begin{tikzcd}[column sep=huge,row sep=huge,baseline=(\tikzcdmatrixname-2-1)]
A \arrow[r,"\e"{name=E}] \arrow[d,"f"] & B \arrow[d,"g"] \\
U \arrow[r,"\f"{name=F}] \arrow[d,"f'"] & V \arrow[d,"g'"] \\
X \arrow[r,"\g"{name=G}] & Y \\
\arrow[Rightarrow,from=F,to=E,"\alpha"]
\arrow[Rightarrow,from=G,to=F,"\beta"]
\end{tikzcd}
=
\begin{tikzcd}[column sep=15em,row sep=huge,baseline=(baseline)]
A \arrow[r,"\e"{name=E}] \arrow[dd,"f' \comp f"{name=baseline}] & B \arrow[dd,"g' \comp g"] \\
\\
X \arrow[r,"\g"{name=F}] & Y
\arrow[Rightarrow,from=F,to=E,"(f' \times g')_*(\alpha) \comp \beta"]
\end{tikzcd}
\]

The source and target functors $S, T: \Dobj \to \Darr$ of the double category are given by taking the left (resp. right) side of a diagram. The unit functor $U: \Dobj \to \Darr$ is given on objects by
\[
	U(A) = \dblhor{A}{A}{i^{1}_{12*}(\O_A)}
\]
and on morphisms by
\[
U\left(\dblver{X}{Y}{f} \right) = \dblcell[column sep=huge,row sep=huge]{X}{X}{Y}{Y}{f}{f}{\O_{\Delta_X}}{\O_{\Delta_Y}}{\mathfrak{u}_f} \,.
\]
The morphism $\mathfrak{u}_f$ is constructed from the sequence of natural transformations
\begin{equation}
\begin{tikzcd}
Y \arrow[rr,equal] & {} \arrow[d,Rightarrow,red,"\text{unit}"] & Y \arrow[r,"i^1_{12}"] & YY \\
Y & X \arrow[l,"f"'] \arrow[r,"f"] & Y\arrow[r,"i^1_{12}"] & YY \\
Y & X \arrow[l,"f"'] \arrow[rr,"i^1_{12} \comp f"{name=F}] &{}& YY \\
Y & X \arrow[l,"f"'] \arrow[r,"i^1_{12}"] & XX \arrow[u,Rightarrow,red,"\text{comp}"] \arrow[r,"f \times f"] & YY
\arrow[from=\tikzcdmatrixname-2-3,to=F,Rightarrow,red,"\text{comp}"]
\end{tikzcd}
\label{eqn:UnitCompositionNatxform}
\end{equation}
Evaluating this natural transformation at $\Oy$ and identifying $f^*(\Oy)$ with $\Ox$ gives a morphism $\mathfrak{u}_f: \O_{\Delta_Y} \to (f \times f)_*(\O_{\Delta_X})$.

The final structure functor needed to construct a double category is the horizontal composition functor $\odot\colon \Darr \times \Darr \to \Darr$. On objects, this is given by
\[
	\dblhor{A}{B}{\e} \odot \dblhor{B}{C}{\f} = \dblhor{A}{C}{\e \odot \f} \,,
\]
where $\e \odot \f = \pi^{ABC}_{AC*}(i^{124*}_{1234}(\e \boxtensor \f))$. On morphisms, the definition is more involved. The composition is given by
\[
	\dblcell{A}{B}{U}{V}{f}{g}{\e}{\p}{\alpha} \fmcomp \dblcell{B}{C}{V}{W}{g}{h}{\f}{\q}{\beta} = \dblcell[column sep=12em]{A}{C}{U}{W}{f}{h}{\e \fmcomp \f}{\p \fmcomp \q}{\eta_{\alpha,\beta} \comp (\alpha \fmcomp \beta)}
\]
where $\eta_{\alpha,\beta}$ is a ``correction morphism'' to account for the fact that $\alpha \fmcomp \beta$ has codomain
\[
	(f \times g)_* (\e) \fmcomp (g \times h)_* (\f) = \pi^{ABC}_{AC*} \comp i^{124*}_{1234}((f \times g)_*(\e) \boxtensor (g \times h)_*(\f)) \,.
\]
Using the distributivity of direct images over $\boxtensor$, the right-hand side is isomorphic to
\[
	\pi^{ABC}_{AC*} \comp i^{124*}_{1234} \comp (f \times g \times g \times h)_*(\e \boxtensor \f)
\]
There is a sequence of natural isomorphisms of functors
\[
\begin{tikzcd}
ABBC \arrow[r, "fg\Id\Id"] & UVBC \arrow[r, "\Id\Id gh"] & UVVW \arrow[d, Rightarrow,"\text{bc}",red] & UVW \arrow[l] \arrow[r] & UW \\
ABBC \arrow[r, "fg\Id_Y\Id_Z"] & UVBC \arrow[d,Rightarrow,red,"\text{counit}"] & UBC \arrow[l,"{\Id(g, \Id)\Id}"'] \arrow[r,"\Id_{U}gh"] & UVW \arrow[r] & UW \\
ABBC \arrow[r,"f\Id\Id\Id"] & UBBC \arrow[d,Rightarrow,red,"\text{bc}"] & UBC \arrow[l,"i^{124}_{1234}"'] \arrow[r,"\Id_{U}gh"] & UVW \arrow[r] & UW \\
ABBC & ABC \arrow[l,"i^{124}_{1234}"'] \arrow[r,"f\Id\Id"] & UBC \arrow[r,"\Id_{U}gh"] & UVW \arrow[r] & UW \\ 
ABBC & ABC \arrow[l,"i^{124}_{1234}"']\arrow[r,"\pi^{ABC}_{AAC}"] & AC \arrow[rr,"fh"{name=F}] & {} & UW \,.
\arrow[from=\tikzcdmatrixname-4-4,to=F,Rightarrow,red,"\text{comp}"]
\end{tikzcd}
\]
Applying the final functor to $\e \boxtensor \f$ gives $(f \times h)_*(\e \fmcomp \f)$, which is the desired codomain. We take $\eta_{\alpha,\beta}$ to be the composition of this sequence of morphisms.

The unitor natural isomorphisms $\mathfrak{l}\colon U_B \tensor M \to M$ and $\mathfrak{r}\colon M \tensor U_A \to M$ are given by the sequence of base-change natural transformations that are used to show that $\Od$ behaves as the identity under composition of Fouier-Mukai kernels (see e.g. \cite[Example~5.4]{Huybrechts06}). For any horizontally composable sequence
\[
\dblhor{A}{B}{\e}, \dblhor{B}{C}{\f}, \dblhor{C}{D}{\g}
\]
of objects in $\Darr$, the associator natural isomorphism is a $2$-cell of the form
\[
	\dblcell[column sep=huge]{A}{D}{A}{D}{\Id_A}{\Id_D}{(\e \fmcomp \f) \fmcomp \g}{\e \fmcomp (\f \fmcomp \g)}{\alpha_{\e,\f,\g}} \,.
\]
Now
\begin{align}
\e \fmcomp (\f \fmcomp \g)
	&= \pi^{126}_{16*} ( i^{126*}_{1,2,36}(\e \boxtensor (\f \fmcomp \g)) \nonumber \\
	&= \pi^{126}_{16*} \comp i^{126*}_{1,2,36} \big(\e \boxtensor \pi^{346}_{36*} \comp i^{346*}_{3,45,6} (\f \boxtensor \g) \big) \nonumber \\
	&\isomorphic \pi^{126}_{16*} \comp i^{126*}_{1,23,6} \comp \pi^{12346}_{1236*} \comp i^{12346*}_{1,2,3,45,6} (\e \boxtensor (\f \boxtensor \g)) \label{eqn:ConstructingDoubleCategoriesProof1}
\end{align}
where in the final line we use the natural isomorphisms from the definition of a monoidal structure on a geofibered category. Using the notation for standard geometric functors, this can be draw diagramatically as the result of applying the functor
\begin{equation}
\begin{tikzcd}
	ABBCCD & ABBCD \arrow[l,"i^{12346}_{1,2,3,45,6}"] \arrow[r,"\pi^{12346}_{1236}"] & ABBD & ABD \arrow[l,"i^{126}_{1,23,6}"] \arrow[r,"\pi^{126}_{16}"]& AD
\end{tikzcd}
\label{eqn:ConstructingDoubleCategoriesProof2}
\end{equation}
to the object $\e \boxtensor (\f \boxtensor \g)$. There is a natural isomorphism
\[
\begin{tikzcd}
	ABBCCD & ABBCD \arrow[l,"i^{12346}_{1,2,3,45,6}"] \arrow[r,"\pi^{12346}_{1236}"] & ABBD \arrow[d,Rightarrow,red,"\text{bc}"] & ABD \arrow[l,"i^{126}_{1,23,6}"] \arrow[r,"\pi^{126}_{16}"]& AD \\
	ABBCCD & ABBCD \arrow[l,"i^{12346}_{1,2,3,45,6}"] & ABCD \arrow[l,"i^{1246}_{1,23,4,6}"] \arrow[r,"\pi^{1246}_{126}"] & ABD \arrow[r,"\pi^{126}_{16}"] & AD \\
	ABBCCD && ABCD \arrow[ll,"i^{1246}_{1,23,45,6}"'{name=E}] \arrow[rr,"\pi^{1246}_{16}"{name=F}] && AD
\arrow[from=\tikzcdmatrixname-2-2,to=E,Rightarrow,red,"\text{comp}"]
\arrow[from=\tikzcdmatrixname-2-4,to=F,Rightarrow,red,"\text{comp}"]
\end{tikzcd}
\]
Combining this with \autoref{eqn:ConstructingDoubleCategoriesProof1}, we find
\[
\e \fmcomp (\f \comp \g) \isomorphic \pi^{1246}_{16*} \comp i^{1246*}_{1,23,45,6} (\e \boxtensor (\f \boxtensor \g)) \,.
\]
Applying a similar process, we also find
\[
(\e \fmcomp \f) \fmcomp \g \isomorphic \pi^{1246}_{16*} \comp i^{1246*}_{1,23,45,6} ((\e \boxtensor \f) \boxtensor \g) \,.
\]
Combining these two results gives the morphism $\alpha_{\e,\f,\g}$.

We now give the data for the symmetric monoidal structure. The category $\Dobj$ can be made into a symmetric monoidal category by taking
\[
	A \tensor B = A \times B \,.
\]
The tensor product of two objects in $\Darr$ is given by the external tensor product,
\[
	\dblhor{A}{B}{\e} \tensor \dblhor{C}{D}{\f} = \dblhor{A \times C}{B \times D}{\e \boxtensor \f} \,.
\]
On morphisms, the tensor product is given by taking the external product, then composing with the distributivity morphism:
\[
	\dblcell{A}{B}{U}{V}{f}{g}{\e}{\p}{\alpha}
		\tensor \dblcell{C}{D}{W}{X}{h}{i}{\f}{\q}{\beta}
	= \dblcell{A \times C}{B \times D}{U \times W}{V \times X}{f \times h}{g \times i}{\e \boxtensor \f}{\p \boxtensor \q}{\alpha \tensor \beta}
\]
where
\[
\alpha \tensor \beta = \mathfrak{d}_{fg,hi} \comp (\alpha \boxtensor \beta) \,.
\]
and
\[
	\mathfrak{d}_{fg,hi}\colon (f \times g)_* \boxtensor (h \times i)_* \isomorphic (f \times g \times h \times i)_* \comp \boxtensor \,.
\]
The associativity natural transformation for this is given by the associator $\alpha_\boxtensor$ for the external product.
\end{proof}

\subsection{Underlying horizontal \texorpdfstring{$2$}{2}-category}

We want to be able to use the monoidal structure on a double category to construct a monoidal structure on its horizontal $2$-category. Hansen and Shulman~\cite{HansenShulman19} provide a condition to permit this construction.

\begin{theorem}\cite{HansenShulman19}
If $\mathbb{D}$ is a monoidal double category, of which the monoidal constraints have
loosely strong companions, then its horizontal $2$-category is a monoidal bicategory.
If $\mathbb{D}$ is braided or symmetric, so is its horizontal $2$-category.
\end{theorem}

Recall a transformation $\alpha$ has loosely strong companions if each component $\alpha_A$ has a loose companion, and the resulting colax transformation $\hat{\alpha}$ is pseudo-natural.

\begin{corollary}
\label{cor:HorizontalCategoryIsMonoidal}
The horizontal $2$-category of the double category constructed in \autoref{prop:ConstructingDoubleCategory} has a symmetric monoidal product, where the monoidal product is given on objects by
\[
	X \tensor Y = X \times Y
\]
and on hom-categories by the functor
\begin{align*}
	\boxtensor\colon \Hom_\Spaces(A, U) \times \Hom_\Spaces(B, V) &\to \Hom_\Spaces(A \tensor U, B \tensor V) \\
		(\e, \f) &\mapsto \e \boxtensor \f \\
		(f, g) &\mapsto f \boxtensor g \,.
\end{align*}
\end{corollary}

\begin{proof}
First, we construct companions for any isomorphism $f\colon A \to B$. Let
\[
\hat{f} = \dblhor[column sep=huge]{A}{B}{(\Id_A, f)_*(\O_A)} \,.
\]
Recall the definition \autoref{eqn:UnitCompositionNatxform} of the morphism $\mathfrak{u}_f\colon \O_{\Delta_B} \to (f \times f)_*\O_{\Delta_A}$. There is a natural isomorphism
\begin{equation}
	\gamma_1\colon (f \times f)_* \comp i^1_{12} \isomorphic (f \times \Id)_* (\Id_A, f)_* \,. \label{eqn:CompanionProof1}
\end{equation}
so there is a $2$-cell
\[
	\dblcell[column sep=9em]{A}{B}{B}{B}{f}{\Id_B}{\hat{f}}{\O_{\Delta_B}}{\gamma_{1, \O_A} \comp \mathfrak{u}_f} \,.
\]
We also have the $2$-cell
\[
	\dblcell[column sep=9em]{A}{A}{A}{B}{\Id_A}{f}{\O_{\Delta_A}}{\hat{f}}{\gamma_{2, \O_A}}
\]
where $\gamma_2$ is the natural isomorphism $\gamma_2\colon (\Id, f)_* \isomorphic (\Id \times f)_* \comp i^1_{12*}$. This pair of $2$-cells gives the auxiliary data needed for $\hat{f}$ to be a companion of $f$. Verifying the conditions that these diagrams must satisfy is an easy check using Theorem~2.4. Note that the second of these $2$-cells is always invertible; the first is invertible if $f$ is an isomorphism. Thus $\hat{f}$ is a companion of $f$ as required.

Thus it remains to show that the colax functor $\hat{\alpha}$ is in fact a pseudo-natural transformation; that is, that each $2$-cell $\hat{\alpha}_f$ is invertible. This component is defined~\cite{HansenShulman19} as the composition
\[
	\hat{\alpha}_f =
\begin{tikzcd}
	FA\arrow[r,"U_{FA}"{name=E}]\arrow[d,"\Id_{FA}"] & FA\arrow[r,"Ff"{name=F}] \arrow[d, "\alpha_A"] & FB\arrow[r,"\hat{\alpha_B}"{name=G}] \arrow[d,"\alpha_B"] & GB \arrow[d,"\Id_{GB}"] \\
	FA \arrow[r,"\hat{\alpha_A}"{name=P}] & GA \arrow[r,"Gf"{name=Q}] & GB \arrow[r,"U_{GB}"{name=R}] & GB
\arrow[Rightarrow,from=P,to=E,"\eta_{\hat{\alpha_A}}"]
\arrow[Rightarrow,from=Q,to=F,"\alpha_f"]
\arrow[Rightarrow,from=R,to=G,"\epsilon_{\hat{\alpha_A}}"]
\end{tikzcd}
\]
The left and right cells are invertible if $f$ is an isomorphism. The central cell is invertible if the morphism $\alpha_f$ is. These conditions are all met for the data defining the monoidal constraints, so they have loosely strict companions as required.
\end{proof}

\begin{corollary}
The $2$-category $\Var$ has a symmetric monoidal structure.
\end{corollary}

\begin{proof}
Consider the geofibered category of derived complexes of sheaves over smooth schemes. The category of spaces has a terminal object (the scheme with a single point, $\pt = \Spec(k)$), and the projection maps are flat, and hence form a pull-geolocalising class. The external tensor product gives this geofibered category a monoidal structure, so the result follows from \autoref{prop:ConstructingDoubleCategory} and \autoref{cor:HorizontalCategoryIsMonoidal}
\end{proof}
\section{Frobenius algebra objects}\label{sec:fao}


It is well-known~\cite{Kock2003} that $(1+1)$-TQFTs valued in the category of vector spaces are classified by Frobenius algebras. This proof immediately generalises to the case when the target category is any monoidal category. In this case, TQFTs are classified by \emph{Frobenius algebra objects}.

\begin{definition}\cite{Kock2003} Let $\C$ be a monoidal category with unit $I$. A Frobenius algebra object is a tuple $(A, \mu, \delta, \epsilon, \tau)$ such that:
\begin{enumerate}
    \item the tuple $(A, \mu, \epsilon)$ is a unital monoid with multiplication $\mu\colon A \tensor A \to A$ and unit $\epsilon\colon I \to A$; and
    \item the tuple $(A, \delta, \tau)$ is a counital comonoid with comultiplication $\delta\colon A \to A \tensor A$ and counit $\tau\colon A \to I$; and
    \item the Frobenius identities hold:
    \begin{align} (\Id_A \tensor \mu) \comp (\delta \tensor \Id_A) = \mu \comp \delta = (\mu \tensor \Id_A) \comp (\Id_A \tensor \delta)\,. \label{eqn:frobidentities}\end{align}
\end{enumerate}
\end{definition}

A $(1+1+1)$-TQFT induces a $(1+1)$-TQFT by truncation; that is, by considering $2$-manifolds only up to diffeomorphism, and taking the isomorphism class of the image of this object. We formalise this folklore process below as a first step to investigating the possible construction of a $(1+1+1)$-TQFT.

\begin{definition}Let $2\catname{-Cat}$ be the $1$-category of $2$-categories and functors, and let $\catname{Cat}$ be the $1$-category of $1$-categories and functors. For $\C \in 2\catname{-Cat}$, let  $\HC$ be the $1$-category with the same objects as $\C$, but with $1$-morphisms given by isomorphism classes of $1$-morphisms in $\C$. For a functor $\F\colon \C \to \D \in \Hom_{2\catname{-Cat}}(\C, \D)$, define
\begin{align*}
	\H\F: \HC &\to \H\D \,, \\
		x &\mapsto \F(x) \,, \\
		[f] &\mapsto [\F(f)] \,,
\end{align*}
where $[f]$ denotes the isomorphism class of the morphism $f$. This defines the core truncation functor
\begin{align*}
	\H\colon 2\catname{-Cat} &\to \catname{Cat} \,.
\end{align*}
\end{definition}

\begin{lemma}Let $\C$ be a $2$-category and $\HC$ be its core truncation. A $(1+1+1)$-TQFT $Z\colon \Bord_{1+1+1} \to \C$ induces a $(1+1)$-TQFT valued in the category $\HC$.\end{lemma}

\begin{proof}
The category $\HBord_{1+1+1}$ is a quotient category of $\Bord_{1+1}$ (since $1$-morphisms in this latter category are bordisms considered up to diffeomorphism, and a diffeomorphism induces an invertible $2$-morphism in $\Bord_{1+1+1}$). The composition
\[ \Bord_{1+1} \to \HBord_{1+1+1} \xrightarrow{HZ} \HC \]
gives the induced $(1+1)$-TQFT.
\end{proof}

\subsection{Constructing Frobenius algebra objects}

\begin{lemma}\label{lem:monoidstructure} Let $X$ be a scheme. Recall $i^1_{123}\colon X \to X^3$ is the triagonal map. Let $\Odd = i^1_{123*} (\Ox)$ be the triagonal sheaf. Then $X \in \HVar$ can be given the structure of a monoid object by taking $\Odd$ as the multiplication map and $\O_{\pt \times X}$ to be the unit.\end{lemma}

\begin{remark}\label{rem:triagisdiagpullback}%
It is worth noting the intuition behind choosing the sheaf $\Odd$ to represent the multiplication. Huybrechts~\cite{Huybrechts06} shows that if $f\colon X \to Y$ and $\Gamma_f = (\Id_X, f)\colon X \to X \times Y$, then the Fouier--Mukai transform with kernel ${\Gamma_f}_*(\Ox) \in D(Y \times X)$ is the inverse image functor $f^*$. In particular, the Fouier--Mukai transform associated to $\Odd$ is the inverse image functor $i^{1*}_{12}$.

If $\e, \f \in \D(X)$ then $i^{1*}_{12}(\e \boxtensor \f) = \e \tensor f$, so this choice of multiplication recovers the tensor product on the subcategory $\D(X) \tensor \D(X) \subset \D(X \times X)$. A similar calculation shows that the Fourier--Mukai transform corresponding to the unit is the pullback functor $\pi^{1*}_\emptyset\colon \D(\pt) \to \D(X)$. 
\end{remark}

\begin{proof}
We first prove that the multiplication is unital; that is, that
\begin{equation}
	\Odd \fmcomp (\Od \boxtensor \Ox) \isomorphic \Od
	\label{eqn:MonoidStructureProof1}
\end{equation}
By \autoref{def:MonoidalGeofiberedCat}, there is an isomorphism
\[ i^1_{12*}(\Ox) \boxtensor \Ox \isomorphic i^{13}_{12,3*}(\Ox \boxtensor \Ox) \isomorphic i^{13}_{12,3*}(\Oxx) \,.
\]
Hence the left-hand side of \autoref{eqn:MonoidStructureProof1} can be written as
\[
	i^4_{456*}(\Ox) \fmcomp i^{13}_{12,3*}(\Oxx) = \pi^{1236}_{16*} \comp i^{1236*}_{1,24,35,6}\left( i^4_{456*}(\Ox) \boxtensor i^{13}_{123*}(\Oxx) \right)\,.
\]
Now
\begin{align*}
i^4_{456*}(\Ox) \boxtensor i^{13}_{123*}(\Oxx)
	&\isomorphic (i^4_{456} \times i^{13}_{123})_* (\Ox \boxtensor \Oxx) \\
	&= i^{134}_{123456*}(\O_{X^3})
\end{align*}
so
\begin{align*}
i^4_{456*}(\Ox) \fmcomp i^{13}_{12,3*}(\Oxx)
	&\isomorphic \pi^{1236}_{16*} \comp i^{1236*}_{1,24,35,6} \comp i^{134}_{12,3,456*}(\O_{X^3}) \\
	&\isomorphic \pi^{1236}_{16*} \comp i^{1}_{123456*} \comp i^{1*}_{134}(\O_{X^3}) \\
	&\isomorphic i^1_{16*}(\Ox)
\end{align*}
where in the second line we apply a base-change isomorphism, and in the third the composition of push-forward functors. Thus in $\HVar$, we have $[\Odd] \comp [\Od \boxtensor \Ox] = [\Od]$ as required.

The right unital identity follows similarly, as does associativity.
\end{proof}

\begin{lemma}Let $X$ be a scheme and let $\Odd = i^1_{123*} (\Ox)$ be the triagonal sheaf. Then $X \in \HVar$ can be given the structure of a comonoid object by taking $\Odd$ to be the comultiplication map and $\O_{X \times \pt}$ to be the counit.\end{lemma}

\begin{proof}These calculations can be performed in a manner similar to those in \autoref{lem:monoidstructure}.\end{proof}

\begin{proposition}\label{prop:diagonaltqft} Let $X$ be a scheme over $k$ and $\Odd = i^{1}_{123*}(\Ox)$. Then the tuple $(X, \Odd, \Odd, \O_{\pt \times X}, \O_{X \times \pt})$ is a Frobenius algebra object in $\HVar$.\end{proposition}

In the remainder of the paper, the TQFT corresponding to this Frobenius algebra object will be denoted $Z_X$.

The claim that this construction gives a TQFT was originally given by Roberts and
Willerton~\cite{RobertsWillerton2003}. They expressed the product map as taking the tensor product with the diagonal sheaf and pushing forward; this can be seen to be equivalent to applying $i^{1*}_{12}$ since
\begin{align*}
\pi^{12}_{2*}(\e \tensor \Od) = \pi^{12}_{2*}(\e \tensor i^{1}_{12*}(\Ox))
    &\isomorphic \pi^{12}_{2*} \comp i^{1}_{12*} \comp i^{1*}_{12}(\e)
    = i^{1*}_{12}(\e) \,,
\end{align*}
where the isomorphism follows from the projection formula \autoref{lem:ProjectionFormulaSpecialised}. From \autoref{rem:triagisdiagpullback}, this functor has Fourier--Mukai kernel $\Odd$.

\begin{proof}
We have seen that this tuple gives rise to a commutative monoid and a cocommutative comonoid, so it remains to show that the Frobenius identities hold.

Consider first the composition $\Odd \fmcomp \Odd = i^{1}_{123*}(\Ox) \fmcomp i^4_{456*}(\Ox)$. We have an isomorphism
\[
	i^{1}_{123*}(\Ox) \boxtensor i^4_{456*}(\Ox) \isomorphic i^{14}_{123,456*}(\Oxx) \,.
\]
so the composition is given by
\begin{align*}
	i^{1}_{123*}(\Ox) \fmcomp i^4_{456*}(\Ox)
		&\isomorphic \pi^{13456}_{1346*} \comp i^{12356*}_{1,2,34,5,6} \comp i^{14}_{123,456*}(\Oxx) \\
		&\isomorphic \pi^{13456}_{1346*} \comp i^1_{13456*} \comp i^{1*}_{14}(\Oxx) \\
		&\isomorphic i^1_{1346*}(\Ox)
\end{align*}
where in the second line we use a base-change isomorphism.

We now simplify $(\Od \boxtensor \Odd) \fmcomp (\Odd \boxtensor \Od)$ to the same form. The two sheaves that we are composing are given by
\[
	\Od \boxtensor \Odd \isomorphic i^{01}_{023,14*}(\Oxx)
\]
and
\[
	\Odd \boxtensor \Od \isomorphic i^{89}_{58,679*}(\Oxx) \,.
\]
The external product of these is given by
\[
	(\Od \boxtensor \Odd) \boxtensor (\Odd \boxtensor \Od) \isomorphic i^{0189}_{023,14,58,679*}(\O_{X^4})
\]
Hence the composition is given by
\begin{align*}
(\Od \boxtensor \Odd) \fmcomp (\Odd \boxtensor \Od)
	&\isomorphic \pi^{0123489}_{0189*} \comp i^{0123489*}_{0,1,25,36,47,8,9} \comp i^{0189}_{023,14,58,679*}(\O_{X^4}) \\
	&\isomorphic \pi^{0123489}_{0189*} \comp i^0_{0123489*} \comp i^{0*}_{0189}(\O_{X^4}) \\
	&\isomorphic i^{0}_{0189*}(\Ox)
\end{align*}
where in the second line we use a base-change isomorphism. Hence the left Frobenius identity holds; the right identity follows similarly.
\end{proof}

\subsection{Properties of Frobenius algebra objects}

\begin{remark}
Until this point, all calculations could have been performed using only the information from the geofibered category structure. From now on, we will use results specific to the category $\Var$.
\end{remark}

\begin{lemma}Let $\Sigma_g$ be a genus $g$ surface. Then
\[ Z_X(\Sigma_g) \isomorphic H^*\left(\Od^{\tensor(g+1)}\right) \,, \]
where we view the right-hand side as a chain complex
\[
	H^n(\Od^{\tensor(g+1)}) \xrightarrow{0} H^{n-1}(\Od^{\tensor(g+1)}) \xrightarrow{0} \cdots \xrightarrow{0} H^0(\Od^{\tensor(g+1)}) \,.
\]
\end{lemma}

\begin{proof}
\begin{figure}
\centering
\begin{tikzpicture}[tqft]
        \node[tqft/pair of pants,draw]   (pants0) {};
        \node[tqft/reverse pair of pants,draw,anchor=outgoing boundary 1] (cpants0) at (pants0.incoming boundary 1) {};
        \node[tqft/cup,draw,anchor=incoming boundary 1] (cup0) at (pants0.outgoing boundary 1) {};
        \node[tqft/cup,draw,anchor=incoming boundary 1] (cup1) at (pants0.outgoing boundary 2) {};

        \node[tqft/pair of pants,draw,anchor=outgoing boundary 1,yshift=2em]   (pantsg) at (cpants0.incoming boundary 1) {};
        \node[tqft/reverse pair of pants,draw,anchor=outgoing boundary 1] (cpantsg) at (pantsg.incoming boundary 1) {};
        
        \path (cpants0.incoming boundary 1) -- node[yshift=0.5em] {$\vdots$} (pantsg.outgoing boundary 1);
        \path (cpants0.incoming boundary 2) -- node[yshift=0.5em] {$\vdots$} (pantsg.outgoing boundary 2);
        \node[tqft/cap,draw,anchor=outgoing boundary 1] (cap0) at (cpantsg.incoming boundary 1) {};
        \node[tqft/cap,draw,anchor=outgoing boundary 1] (cap1) at (cpantsg.incoming boundary 2) {};
        
        \path (cpantsg.incoming boundary 1) ++(2,0) node (v0) {};
        \path (pants0.outgoing boundary 1)  ++(2,0) node (v1) {};
        \draw[decorate,decoration={brace,amplitude=10pt}] (v0) -- (v1) node[midway,xshift=4em] {$g+1$ copies};
\end{tikzpicture}
\caption{A decomposition of the genus $g$ surface}
\label{fig:decomposegsurface}

\end{figure}
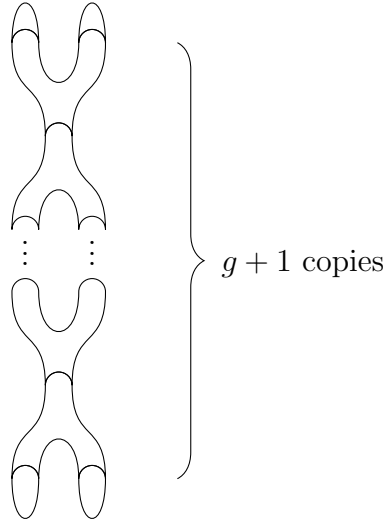

Consider the decomposition of a genus $g$ surface as shown in \autoref{fig:decomposegsurface}. Applying $Z_X$ gives
\[
Z_X(\Sigma_g) = Z_X\left(\bordismtensor[minitqft]{cap}{cap}\right) \comp Z_X\left(\bordismtwo[minitqft]{reverse pair of pants}{pair of pants}\right)^{g+1} \comp Z_X\left(\bordismtensor[minitqft]{cup}{cup}\right) \,.
\]

Now consider the Fourier--Mukai transform $\Phi_{Z_X(\Sigma_g)}$ associated to the kernel $Z_X(\Sigma_g)$. By \autoref{rem:kernelsnotunique}, $Z_X(\Sigma_g) \isomorphic \Phi_{Z_X(\Sigma_g)}(k)$, so to determine $Z_X(\Sigma_g)$ it is sufficient to compute the functor $\Phi_{Z_X(\Sigma_g)}$. Let $\c_1 = Z_X(\bordismtensor[minitqft]{cap}{cap}), \d_1 = Z_X(\bordism[minitqft]{pair of pants}), \d_2 = Z_X(\bordismcopants)$ and $\c_2 = Z_X(\bordismtensor[minitqft]{cup}{cup})$. 
Now
\[ \c_1 = \Ox \boxtensor \Ox = \Oxx \]
and likewise
\[ \c_2 = \Ox \boxtensor \Ox = \Oxx \,; \]
hence the corresponding Fourier--Mukai transforms are
\begin{align*}
\Phi_{\c_1}(\e) &= \pi^{12}_{\emptyset*}(\pi^{12*}_{12}(\e) \tensor \Oxx) = \pi^{12}_{\emptyset*}(\e) \,,\\
\Phi_{\c_2}(\e) &= \pi^{12}_{12*}(\pi^{12*}_{\emptyset}(\e) \tensor \Oxx) = \pi^{12*}_{\emptyset}(\e)\,.
\end{align*}

By \autoref{rem:triagisdiagpullback}, we have $\Phi_{\d_1} = i^1_{12*}$ and similarly $\Phi_{\d_2} = i^{1*}_{12}$. By \autoref{lem:ProjectionFormulaSpecialised},
\[
i^1_{12*} \comp i^{1*}_{12}(\e) \isomorphic i^{1*}_{12}(\e \boxtensor i^1_{12*}(\Ox)) \isomorphic \e \tensor i^1_{12*}(\Ox) \,,
\]
so $\Phi_{\d_1} \comp \Phi_{\d_2}(\e) = \e \tensor i^1_{12*}(\Ox) = \e \tensor \Od$.
Combining these results gives
\begin{align*}
\Phi_{Z(\Sigma_g)}(k) &= \Phi_{\c_1} \comp (\Phi_{\d_1} \comp \Phi_{\d_2})^{g+1} \comp \Phi_{\c_2}(k) \\
	&= \pi^{12}_{\emptyset*}(\pi^{12*}_{\emptyset}(k) \tensor \Od^{\tensor (g+1)}) \\
	&= \pi^{12}_{\emptyset*}(\Od^{\tensor (g+1)}) \,.
\end{align*}
Since the direct image functor with codomain a single point equals the global sections functor, and sheaf cohomology is the derived global sections functor, this expression is exactly $H^*\left(\Od^{\tensor(g+1)}\right)$.
\end{proof}

We now look to give an explicit computation of the derived tensor product of the diagonal sheaf. To do this, we will find a flat resolution of the diagonal. Pragacz, Srinivas and Pati~\cite{PSP2007} give such a resolution in the case where $X$ is a smooth variety which satisfies a certain property.

\begin{definition}\cite{PSP2007} Let $X$ be a smooth variety. We say $X$ has property $(D)$ when there exists a vector bundle $\e$ of rank equal to $\dim(X)$ on $X \times X$ and a section $s$ of $\e$, such that the zero scheme of $s$ is the diagonal subscheme $\Delta \subset X \times X$.
\end{definition}

In particular, if $X$ is a projective surface which is birational to a $K3$ surface with two disjoint rational curves, then $X$ satisfies property $(D)$. The property is also closed under taking fibre products.

\begin{proposition}\label{prop:PropertyDStateSpaceCalculation} Suppose that $X$ is a scheme which has property $(D)$. Then 
\[ Z_X(\Sigma_g) \isomorphic H^*\left(\left({\bigwedge}^* \Omega\right )^{\tensor g}\right) \,, \]
where $\Omega$ denotes the cotangent bundle and $\bigwedge^* \Omega$ is the chain complex with the sheaf $\wedge^i \Omega$ in degree $i$ and zero differential.
\end{proposition}

\begin{proof}
For $g = 0$, we have
\[
Z_X(\Sigma_0) = \pi^{12}_{\emptyset*}(i^{1}_{12*}(\Ox)) = \pi^1_{\emptyset*}(\Ox) \isomorphic H^*(\Ox) \,.
\]
Now consider the case $g \geq 1$. Under the assumption that $X$ has property $(D)$, we can find a Koszul resolution of the diagonal sheaf of the form 
\begin{align} \Od \isomorphic C_\bullet = 0 \to \det(\e^*) \to \cdots \to \e^* \xrightarrow{s^*} \Oxx \to 0 \,, \label{eqn:diagresolution} \end{align}
where $\e$ is a locally free sheaf and $s$ is a section of $\e$ which vanishes on the diagonal of $X \times X$. Here, $\e^*$ is the dual sheaf and $s^*\colon \e^* \to \Oxx$ is the dual map given by evaluating at the section $s$.

Since $C_\bullet$ is a complex of free modules, the derived functor $\Delta_X^*$ is the same as applying the non-derived functor $\Delta_X^*$ termwise. From~\cite{PSP2007} we have $\Delta_X^*(\e^*) = \Omega_X$, so
\begin{align}
\Delta_X^* (C_\bullet) = 0 \to \det(\Omega_X) \to \cdots \to \Omega_X \xrightarrow{0} \Ox \to 0 = {\bigwedge}^* \Omega \label{eqn:diagpullback} \,.
\end{align}
Thus
\[ \Delta_X^*(\Od^{\tensor g}) \isomorphic (\Delta_X^*(\Od))^{\tensor g} \isomorphic \left({\bigwedge}^* \Omega\right)^{\tensor g} \,. \]
By \autoref{lem:ProjectionFormulaSpecialised}, we have
\[ \Od^{\tensor g} \tensor \Od \isomorphic \Delta_{X*} \comp \Delta_X^* (\Od^{\tensor g})\,, \]
so
\begin{align*}
\Od^{\tensor (g+1)}
	&\isomorphic \Delta_{X*} \comp \Delta_X^* (\Od^{\tensor g}) \\
	&\isomorphic \Delta_{X*} \left( \left({\bigwedge}^* \Omega\right)^{\tensor g} \right) \,,
\end{align*}
and hence
\[ Z(\Sigma_g) = \pi^{12}_{\emptyset*}(\Od^{\tensor (g+1)}) = \pi^1_{\emptyset*}\left(\left({\bigwedge}^* \Omega\right)^{\tensor g}\right) = H^*\left(\left({\bigwedge}^* \Omega\right)^{\tensor g}\right) \]
as required.
\end{proof}

\begin{corollary}Under the hypothesis of \autoref{prop:PropertyDStateSpaceCalculation}, the state spaces of the TQFT $Z_X$ are isomorphic to those of the Rozansky--Witten TQFT.\end{corollary}

\section{The affine subcategory \texorpdfstring{$\AffVar$}{AffVar}}\label{sec:AffineSubcategory}

Consider the full subcategory $\AffVar \subset \Var$ whose objects are affine schemes over $k$. Recall that there is a simple description of quasi-coherent sheaves of modules over an affine scheme $\Spec(R)$: they are determined by their global sections, which are exactly $R$-modules. We use this to give an equivalence between the categories $\AffVar$ and the category $\DAlg$ formed of algebras, complexes of bimodules and chain maps (see \autoref{def:DAlg} for the full definition).

\subsection{Non-derived equivalence} \label{sec:AffineSubcategoryNonDerived}

Let $R$ be a $k$-algebra. Recall the global sections functor
\[ \Gamma\colon \Sh(\Spec(R)) \to \Rmod \,\]
sends a sheaf to its global sections. This functor has a quasi-inverse functor
\[ \widetilde{\hphantom{M}}\colon \Rmod \to \Sh(\Spec(R)) \,\]
which sends a module $M$ to the $\O_{\Spec(R)}$-module $\widetilde{M}$. This gives an equivalence of categories~\cite{Hartshorne1977}
\begin{equation}
\begin{tikzcd}
R\catname{-mod}
\arrow[r, "\widetilde{\hphantom{M}}"{name=F}, yshift=1mm] &
\Sh(\Spec(R)) \,,
\arrow[l, "\Gamma"{name=G}, yshift=-1mm]
\end{tikzcd} \label{eqn:GlobalSectionsTildeEquivalence}
\end{equation}
which also gives an equivalence of the derived categories.

It will be useful to translate tensor product functor and the direct and inverse image functors to this algebraic setting.

\begin{lemma}\label{lem:GlobalSectionsDistributeOverTensor} The global sections functor distributes over the tensor product; that is, $\Gamma(\e \tensor \f) \isomorphic \Gamma(\e) \tensor \Gamma(\f)$. Similarly, $\widetilde{M \tensor N} \isomorphic \widetilde{M} \tensor \widetilde{N}$.
\end{lemma}

\begin{definition}\label{def:DirectInverseImageFunctorsRings}
Let $f^\#\colon R \to S$ be a ring morphism, and let $f\colon \Spec(S) \to \Spec(R)$ be the induced morphism of schemes. Define
\begin{align*}
	f^\#_*\colon \Smod &\to \Rmod\,, \\
	M &\mapsto M_R \,,
\end{align*}
where $M_R$ denotes the $S$-module $M$ considered as an $R$-module by $r s = f(r)s$, and
\begin{align*}
	f^{\#*}\colon \Rmod &\to \Smod\,, \\
	N &\mapsto N \tensor_R S \,,
\end{align*}
where $R$ acts on $S$ by $rs = f^\#(r)s$.
\end{definition}

\begin{lemma}\label{lem:AffineImageFunctorsCorrespondance}\cite{Hartshorne1977}
The functors $f^{\#*}$ and $f^\#_*$ correspond to the inverse image and direct image functors respectively; explicitly, we have commutative diagrams of functors
\[
\begin{tikzcd}
\Sh(\Spec(S)) \arrow[r,"f_*"] \arrow[d,"{\Gamma}"] & \Sh(\Spec(R)) \arrow[d,"{\Gamma}"] \\
\Smod \arrow[r, "f^\#_*"] & \Rmod
\end{tikzcd}
\]
and
\[
\begin{tikzcd}
\Sh(\Spec(R)) \arrow[r,"f^*"] \arrow[d,"{\Gamma}"] & \Sh(\Spec(S))\arrow[d,"{\Gamma}"] \\
\Rmod \arrow[r, "f^{\#*}"]  & \Smod  \\
\end{tikzcd} \,.
\]
\end{lemma}

\subsection{\texorpdfstring{$2$}{2}-categorical equivalence}

Recall\cite{BDSV2015} the symmetric monoidal $2$-category $\Alg$ is defined to have objects $k$-algebras and categories of $1$-morphisms given by
\[ \Hom_\Alg(A, B) = \ABBimod, \]
where $\ABBimod$ is the category of $A$-$B$-bimodules. The composition in this category is given by relative tensor product: given $M \in \Hom_\Alg(A, B)$, $N \in \Hom_\Alg(B, C)$ their composition is $M \comp N = M \tensor_B N$. The monoidal structure is the tensor product over $k$.

Since we are working with derived categories, we want to consider a derived version of this category. Since $\ABBimod$ is equivalent to $(A \tensor_k B^\text{op})\catname{-mod}$, which in turn (since $B$ is commutative) is isomorphic to the category $\ABmod$, we replace $\ABBimod$ with $\DABBimod$.

\begin{definition}\label{def:DAlg}
Let $\DAlg$ be the $2$-category with the same objects as $\Alg$, and with categories of $1$-morphisms
\[ \Hom_\DAlg(A, B) = \DABBimod \,. \]
Vertical composition is the usual composition of morphisms in $\DABBimod$. The composition functor is given by the derived tensor product
\[ \tensor_B\colon \D((B \tensor_k C)\catname{-mod}) \times \D((A \tensor_k B)\catname{-mod}) \to \D((A \tensor_k C)\catname{-mod}) \,, \]
where we view an $(A \tensor_k B)$-module as a $B$-module by the inclusion of algebras $B \hookrightarrow A \tensor_k B$ (and likewise for $(B \tensor_k C)$-modules). The monoidal structure is given by taking the tensor product over $k$.
\end{definition}

We want to use the equivalence of $1$-categories in \autoref{eqn:GlobalSectionsTildeEquivalence} to construct an equivalence of $2$-categories between $\DAlg$ and $\AffVar$. Let $\AlgInclFn\colon \DAlg \to \AffVar$ be defined on objects by $\AlgInclFn(A) = \Spec(A)$. On $\Hom$-categories, define
\begin{align*}
	\AlgInclFn\colon \Hom_\DAlg(A, B)  &\to \Hom_\AffVar(\AlgInclFn(A), \AlgInclFn(B)) 
\end{align*}
to be the derived functor of the (exact) functor $\sheaffunctor$.

\begin{lemma}The map $\Phi$ is a symmetric monoidal $2$-functor.\end{lemma}

\begin{proof}
We will show that $\Phi$ respects the identity morphism and composition of morphisms in $\DAlg$ up to $2$-morphism. For $A \in \DAlg$ and $X = \Spec(A)$, we have 
\[ \Gamma(\Od) = A \,,\]
where $A$ has the structure of an $(A \tensor_k A)$-bimodule where the left and right actions are both multiplication. This is exactly the identity morphism $\Id_A \in \Hom_\DAlg(A, A)$. Hence there is an isomorphism $\Od \isomorphic \widetilde{\Gamma(\Od)} = \widetilde{\Id_A}$.

Now let $M \in \Hom_\DAlg(A, B)$, $N \in \Hom_\DAlg(B, C)$. For brevity of notation, let $AB = A \tensor_k B$ and likewise for $AC, BC$ and $ABC$. Let $X_1 = \Spec(A), X_2 = \Spec(B), X_3 = \Spec(C)$. Following the notation of \autoref{def:DirectInverseImageFunctorsRings}, we have
\begin{align*}
\pi^{123\#}_{12}\colon AB &\to ABC\,, \\
				  a \tensor b &\mapsto a \tensor b \tensor 1_C\,,
\end{align*}
and can calculate $\pi^{123\#}_{13}$ and $\pi^{123\#}_{23}$ similarly. Then 
\begin{align*}
\Gamma\left(\Phi(N) \comp \Phi(M)\right)
	&= \Gamma \comp \pi^{123}_{13*}\left(\pi^{123*}_{23}(\Phi(N)) \tensor \pi^{123*}_{12}(M)\right) \\
	&= \pi^{123\#}_{13*} \left(\pi^{123\#*}_{23}(\Gamma \comp \Phi(N)) \tensor \pi^{123\#*}_{12}(\Gamma \comp \Phi(M))\right) \\
	&= \left((N \tensor_{BC} ABC) \tensor_{ABC} (M \tensor_{AB} ABC) \right)_{AC} \,,
\end{align*}
where in the second line we used \autoref{lem:AffineImageFunctorsCorrespondance} and \autoref{lem:GlobalSectionsDistributeOverTensor}.

Let $N_\bullet$ be a free resolution of $N$ and $M_\bullet$ be a free resolution of $M$ as $ABC$-modules. Then
\[ N_\bullet \tensor_{BC} ABC = N_\bullet \tensor_k A \]
and
\[ M_\bullet \tensor_{AB} ABC = M_\bullet \tensor_k C \,, \]
so
\begin{align*}
\Gamma(\Phi(N) \comp \Phi(M))
	&\isomorphic \left( (N_\bullet \tensor_{BC} ABC) \tensor_{ABC} (M_\bullet \tensor_{AB} ABC) \right)_{AC} \\
	&\isomorphic \left((N_\bullet \tensor_k A) \tensor_{ABC} (M_\bullet \tensor_k C) \right)_{AC} \\
	&= \left(N_\bullet \tensor_B M_\bullet \right)_{AC} \,.
\end{align*}
On the other hand,
\begin{align*}
N \comp M
	&= (N \tensor_B M)_{AC} \\
	&\isomorphic (N_\bullet \tensor_B M_\bullet)_{AC} \,.
\end{align*}

Thus $\Gamma(\Phi(N) \comp \Phi(M)) \isomorphic N \comp M$. Applying $\Phi$ to both sides and using the fact that $\Phi$ and $\psi$ are quasi-inverse, we find $\Phi(N) \comp \Phi(M) \isomorphic \Phi(N \comp M)$ as required.

We now show that $\Phi$ is monoidal. Let $A, B \in \DAlg$. Then
\[ \Phi(A \tensor B) = \Spec(A \tensor B) \,, \]
while
\[ \Phi(A) \tensor \Phi(B) = \Spec(A) \times \Spec(B) \isomorphic \Spec(A \tensor B)  \]
as required. Finally, by \autoref{lem:GlobalSectionsDistributeOverTensor} we find $\widetilde{M \tensor N} \isomorphic \widetilde{M} \tensor \widetilde{N}$.
\end{proof}

We now construct a quasi-inverse to $\AlgInclFn$. Let $\AlgProjFn\colon \AffVar \to \DAlg$ take a scheme to the global sections of its structure sheaf. On morphisms, we take this to be the derived functor of the (exact) functor $\Gamma$.

On objects, $\AlgInclFn \comp \AlgProjFn(X) = \Spec(\Gamma(X)) \isomorphic X$, and $\AlgProjFn \comp \AlgInclFn(A) = A$. Further, the functors $\AlgInclFn_{AB}$ and $\AlgProjFn_{\Phi(A), \Phi(B)}$ define equivalences of categories. This, along with the fact that $\AlgInclFn$ is a $2$-functor, is enough for $\C$ and $\D$ to be equivalent $2$-categories. Formally, we have the following folklore result:

\begin{lemma}Let $\C, \D$ be $2$-categories, with $\Phi\colon \C \to \D$ a $2$-functor. Suppose that $\Phi$ is surjective up to invertible $1$-morphisms; that is, for any $d \in \D$ there is some $c \in \C$, $f \in \Hom_\D(\Phi(c), d)$, $g \in \Hom_D(d, \Phi(c))$ and an invertible $2$-morphism $fg \Rightarrow \Id_d$. Suppose further that
\[ \Phi_{AB} \colon \Hom_\C(A, B) \to \Hom_\D(\Phi(A), \Phi(B)) \]
is an equivalence of categories for all $A, B \in \C$. Then $\C, \D$ are equivalent $2$-categories.
\end{lemma}

This is a adaptation of the statement that a functor between $1$-categories is an equivalence if and only if it is fully faithful and essential surjective.

\section{Extended TQFTs valued in \texorpdfstring{$\Var$}{Var}}

\subsection{Affine extended TQFTs}\label{sec:AffineETQFT}

In this section, we will prove the result of \autoref{thm:ETQFTDiscrete} under the additional assumption that $X = Z(S^1)$ is affine and irreducible.

In $\Bord_{1+1+1}$, the cap $\bordism[minitqft]{cap}$ and cup $\bordism[minitqft]{cup}$ bordisms are an adjoint pair. Following the notation of \cite{BDSV2015}, the unit and counit are given by the morphisms $\mu^\dagger$ and $\nu^\dagger$, defined as follows. The morphism
\[
\mu^\dagger \colon
\begin{tikzpicture}[minitqft,baseline=(baselinecoord)]
	\node [tqft/cylinder,tqft/cobordism height=2.4em,draw] (cyl0) {};
	\path [blue,draw] (cyl0.east) to[bend right=30] (cyl0.west);
	\path [blue,dashed,draw] (cyl0.east) to[bend left=30] (cyl0.west);
	\coordinate (baselinecoord) at ([yshift=-1mm]cyl0.west);
\end{tikzpicture}
\to
\begin{tikzpicture}[minitqft,baseline=(baselinecoord)]
	\node[tqft/cup,draw] (cup0) {};
	\node[tqft/cap,draw,anchor=incoming boundary 1] (cup1) at (cup0.outgoing boundary 1)  {};
	\coordinate (baselinecoord) at ([yshift=-1mm]cup0.outgoing boundary 1);
\end{tikzpicture}
\,
\]
is the trace of surgery around the blue circle. The morphism
\[
\nu^\dagger \colon \bordismtwo[minitqft]{cap}{cup} \to \bordismbox[minitqft]
\]
is given by the trace of surgery on the $2$-sphere.

Since these are a unit-counit pair, we have the relation
\begin{equation}
\begin{tikzpicture}[halftqft,baseline={([yshift=0.6em]cap0.incoming boundary 1)}]
	\node [tqft/cylinder,draw,tqft/cobordism height=2.4em] (cyl0) {};
	\node [tqft/cup,draw,anchor=incoming boundary 1] (cap0) at (cyl0.outgoing boundary 1) {};
\end{tikzpicture}
\xrightarrow{\mu^\dagger \comp \bordism[minitqft]{cup}}
\begin{tikzpicture}[halftqft,baseline={([yshift=0.6em]cup1.incoming boundary 1)}]
	\node [tqft/cup,draw] (cup0) {};
	\node [tqft/cap,draw,anchor=incoming boundary 1] (cap0) at ([yshift=2.4em]cup0.outgoing boundary 1) {};
	\node [tqft/cup,draw,anchor=incoming boundary 1] at (cap0.outgoing boundary 1) (cup1) {};
\end{tikzpicture}
\xrightarrow{\bordism[minitqft]{cup} \comp \nu^\dagger}
\begin{tikzpicture}[halftqft,baseline={([yshift=0.6em]cup1.incoming boundary 1)}]
	\node [tqft/cup,draw] (cup0) {};
	\node [tqft/cap,anchor=incoming boundary 1] (cap0) at ([yshift=2.4em]cup0.outgoing boundary 1) {};
	\node [tqft/cup,anchor=incoming boundary 1] at (cap0.outgoing boundary 1) (cup1) {};
	\node [coordinate] (topleft)  at ([xshift=-6pt,yshift=-1.5em]cap0.north) {};
	\node [coordinate] (botright) at ([xshift= 6pt,yshift= 1.2em]cup1.south) {};
	\node [fit=(topleft)(botright),rectangle, draw, dotted] (box0) {};
\end{tikzpicture}
=
\bordism[halftqft]{cup}
\xrightarrow{\Id}
\bordism[halftqft]{cup} \,.
\label{eqn:MuNuComposition}
\end{equation}
where we use whiskering to compose the $1$-morphism $\bordism[minitqft]{cup}$ with the $2$-morphisms $\mu^\dagger$ and $\nu^\dagger$ (recall in the bordism category, this whiskering is performed by taking the product of the $1$-morphism with the unit interval $I$ to give the identity $2$-morphism over it, and then gluing the $2$-morphisms together along their boundary). In particular, $Z(\mu^\dagger) \comp Z(\bordism[minitqft]{cup})$ is a monomorphism.

Suppose $Z$ is a $(1+1+1)$-TQFT such that the induced $(1+1)$-TQFT is $Z_X$ for some $X$. Then there are isomorphisms
\[ Z(\bordism[minitqft]{cylinder}) \isomorphic \Od \]
and
\[ Z(\bordismcapcup) \isomorphic \Oxx \,, \]
so we have
\[ Z(\mu^\dagger) \in \Hom(Z(\bordism[minitqft]{cylinder}), Z(\bordismcapcup)) \isomorphic \Hom(\Od, \Oxx) \,. \]
We will show that $\Hom(\Od, \Oxx) = \setdef{0}$ unless $X$ is a field.

\newcommand{\Fm}{{\mathbb{F}_\m}}
\newcommand{\Fn}{{\mathbb{F}_\n}}
\begin{definition}
For a ring $R$ and a maximal ideal $\m \in \Specm(R)$, let $\Fm = R/\m$ be the residue field. Let $\Ev_{\m}\colon R \to \Fm$ be the map formed by taking the quotient of $R$ by $\m$. Define $\Ev_{\m,\n}\colon R \tensor_k R \to \Fm \tensor \Fn$ by
\[ \Ev_{\m,\n} = (\Id_\Fm \tensor \Ev_\n) \comp (\Ev_m \tensor \Id_R) \,.\]
\end{definition}

\begin{lemma}\label{lem:zeroonepointimpliesdirectsum}
Let $R$ be a ring with $\jacobsonradical(R) = 0$ and $\m \ideal R$ be a fixed maximal ideal. If
\[ J_\m = \bigintersect_{\n \in \Specm(R), \n \neq \m} \n \neq \setdef{0}\,, \]
then $R = \m \directsum J_\m$.
\end{lemma}

\begin{proof}Pick some $r \in J_\m \setminus \setdef{0}$. Then
\[ J_\m \intersect \m = \jacobsonradical(R) =\setdef{0}\,, \]
so $r \not\in \m$. Hence $\m \subsetneq \m + J_\m$, but since $\m$ is a maximal ideal we must have $\m + J_\m = R$. Since $J_\m \intersect \m = \jacobsonradical(R) = \setdef{0}$ we have $R = \m \directsum J_\m$.
\end{proof}

\begin{lemma}\label{lem:PropertiesOfMapsStoD}
Let $R$ be an algebra with $\jacobsonradical(R) = \setdef{0}$. Suppose $R$ cannot be written as a non-trivial direct sum. Let $S = R \tensor_k R$ and let $D$ be the $S$-module with underlying abelian group $R$ and with $r \tensor r' \in S$ acting by multiplication by $rr'$. Then 
\[ \Hom_S(D, S) = \begin{cases} R & R\text{ is a field} \\ \setdef{0} & \text{otherwise} \end{cases} \]
\end{lemma}

\begin{proof}
Say $g \in \Hom_S(D, S)$ and let $g_1 = g(1 \tensor 1)$. Fix maximal ideals $\m, \n \in \Specm(R)$, $\m \neq \n$. Then $\m \setminus \n \neq 0$ (else $\m \subset \n$ and since $\m$ is maximal this would give $\m = \n$) so pick some $r_0 \in \m \setminus \n$. Then $\Ev_\m(r_0) = 0$, $\Ev_\n(r_0) \neq 0$. Now
\[ 0 = (r_0 \tensor 1 - 1 \tensor r_0) \cdot (1 \tensor 1) \in D  \,, \]
so
\[ 0 = g(0) = (r_0 \tensor 1 - 1 \tensor r_0) g_1 \in S \,. \]
Applying $\Ev_{\m,\n}$ gives
\[ 0 = \Ev_{\m, \n}\left((r_0 \tensor 1 - 1 \tensor r_0) g_1\right) = (-1 \tensor \Ev_\n(r_0)) \Ev_{\m,\n}(g_1)\,, \]
and hence $\Ev_{\m, \n}(g_1) = 0$. Let $g_\m = (\Ev_\m \tensor \Id_R)(g_1)$, so $(\Id_\Fm \tensor \Ev_\n)(g_\m) = 0$ and hence $g_\m \in \Fm \tensor \n$. As this holds for all $\n \neq \m$, we have $g_\m \in \Fm \tensor J_\m$.

Suppose first $g_\m \not\in \Fm \tensor \m$. Then there is some $r_\m \in J_\m \setminus \m$ (and in particular $J_\m \neq \setdef{0}$). By \autoref{lem:zeroonepointimpliesdirectsum}, $R$ can be written as a direct sum $R = \m \directsum J_\m$. By assumption this direct sum must be trivial, so $\m = 0$ and $R$ is a field.

Suppose instead $g_\m \in \Fm \tensor \m$ for all $\m \in \Specm(R)$. Then $g_\m \in \Fm \tensor \jacobsonradical(R) = 0$, and hence
\[ g_1 \in \bigcap_{\m \in \Spec(R)} \ker(\Ev_\m \tensor \Id_R) = \bigcap_{\m \in \Spec(R)} \m \tensor R = \jacobsonradical(R) \tensor R = \setdef{0} \]
and so $g = 0$.
\end{proof}

\begin{corollary}\label{lem:HomSetDToS}
Let $X$ be an affine irreducible scheme. Then
\[
\Hom_\AffVar(\Od, \Oxx) = 
\begin{cases}
	\mathbb{F} & X \isomorphic \Spec(\mathbb{F})\text{ for some field }\mathbb{F} \\
	0 & \text{otherwise}
\end{cases}
\]
\end{corollary}

We are now in position to prove \autoref{thm:ETQFTDiscrete} under the assumption that $Z(S^1)$ is affine and irreducible.

\begin{proof}[Partial proof of \autoref{thm:ETQFTDiscrete}] \label{sec:ProofAffineETQFTDiscrete}
Let $R$ be such that $X = \Spec(R)$. Then $Z(\mu^\dagger) \in \Hom_\AffVar(\Od, \Oxx)$. Since $Z(\mu^\dagger) \comp Z(\bordism[minitqft]{cup})$ is the unit of an adjunction, it must be non-zero, so by \autoref{lem:HomSetDToS} $R$ is a field.
\end{proof}

\subsection{Extended TQFTs}

Let $Z$ be a TQFT valued in $\Var$. We say that $Z$ is \emph{based} on the scheme $X = Z(S^1)$. Given such a TQFT, it is natural to ask if we can use this to construct a TQFT based on an open affine subscheme $U \subset X$.

Let $\iota\colon U \hookrightarrow X$. Then we have $(\iota^l)^*\colon \D(X^l) \to \D(U^l)$, and we can use this to construct functors
\[ \iota^{m+n}\colon \Hom_\Var(X^m, X^n) = \D(X^{m+n}) \to \D(U^{m+n}) = \Hom_\Var(U^m, U^n) \,.\] 
However, this need not respect composition of morphisms. For example, take $X = \mathbb{P}^1$ to be the projective line and $U$ to be an affine patch. Let $\e = \Ox \in \Hom(\pt, X)$ and $\f = \Ox \in \Hom(X, \pt)$. Then $(\iota^0)^*(\e \comp \f) = \pi^{1}_{\emptyset*}(\Ox) = k \directsum k$, but $(\iota^1)^*(\Ox) \tensor (\iota^1)^*(\Ox) = \pi^1_\emptyset\O_U = k$.

However, the inverse image functors $\iota^*$ do allow us to partially construct a TQFT based on $U \subset X$. This idea allows us to prove the main theorem.

\begin{theorem}
Let $Z$ be a $(1+1+1)$-TQFT valued in $\Var$ such that the induced $(1+1)$-TQFT corresponds to the Frobenius algebra object in \autoref{prop:diagonaltqft}, where $X = Z(S^1)$. If $X$ is of finite type and reduced, then it must be discrete. In this case, $Z$ is isomorphic to a direct sum of extended TQFTs, each of which sends $S^1$ to a single point.
\end{theorem}

\begin{proof}
Let $U \subset X$ be an irreducible open affine subset and $\iota\colon U \to X$ be the inclusion map. There is a fibre square
\[
\begin{tikzcd}
	U \arrow[r,"i^1_{12}"] \arrow[d,"\iota"] & U^2\arrow[d,"\iota^2"] \\
	X \arrow[r,"i^1_{12}"] & X^2
\end{tikzcd}
\]
where in the top line we abuse notation and use the same notation for the map $i^1_{12}: U \to U^2$ as the map $i^1_{12}: X \to X^2$. Using the base-change formula, we see
\[ \iota^{2*} \comp i^1_{12} \isomorphic i^1_{12*} \comp \iota^*\,, \]
so $\iota^{2*}(\Od) = i^1_{12*}(\O_U)$. Thus we have
\begin{align*}
\iota^{2*} \comp Z(\mu^\dagger)
	&\in \Hom(i^1_{12*}(\O_U), \O_{U^2}) \\
	&\isomorphic \Hom(Z_U(\bordism[minitqft]{cylinder}), Z_U(\bordismcapcup))
	\,.
\end{align*}
In particular, we are in the situation to apply \autoref{lem:HomSetDToS}: either $U$ is a point, or $\iota^{2*} \comp Z(\mu^\dagger) = 0$.

Let $f = Z(\mu^\dagger)$. Then
\begin{align*}
Z(\mu^\dagger \comp \bordism[minitqft]{cup})
	&= \pi_{2*}\left( \pi_1^*(\Ox) \tensor f\right) \\
	&= \pi_{2*} \left(f\right) \,.
\end{align*}
Since $Z(\mu^\dagger \comp \bordism[minitqft]{cup})$ is a monomorphism, and $\iota^*$ is exact, we see $\iota^* \comp Z(\mu^\dagger \comp \bordism[minitqft]{cup}) = \iota^* \comp \pi_{2*}(f)$ is also a monomorphism.

Let $f = Z(\mu^\dagger) \in \Hom(\Od, \Oxx)$. Suppose that $\iota^{2*}(f) = 0$. Then since $\Od$ is supported on the diagonal, $f$ is determined by a neighbourhood of the diagonal, so $(\Id \times \iota)^*(f) = 0$. There is a fibre diagram
\[
\begin{tikzcd}
	X \times U \arrow[r,"\pi_2"] \arrow[d,"\Id \times \iota"] & U \arrow[d,"\iota"] \\
	X \times X \arrow[r,"\pi_2"] & X \,.
\end{tikzcd}
\]
This gives a natural isomorphism $\pi_{2*} \comp (\Id \times \iota) \isomorphic \iota^* \comp \pi_{2*}$, and hence $\iota^* \comp \pi_{2*}(f) = 0$, giving a contradiction.

Thus we have $\iota^{2*}(f) \neq 0$. Since $U$ is an irreducible affine open, by \autoref{lem:HomSetDToS} we must have $U = \Spec(\mathbb{F})$ for some field $\mathbb{F}$. This holds for any irreducible affine open $U$, so $X$ is union of discrete points as required.
\end{proof}

\subsection{Chain maps of non-zero degree}

A natural modification of the category $\Var$ would be to allow the $2$-morphisms to be formed from chain maps with possibly non-zero degree (that is, elements of $\Ext(\e, \f)$, rather than $\Hom(\e, \f)$). However, the map $\mu^\dagger$ can be seen to have degree $0$ as follows. The composition in \autoref{eqn:MuNuComposition} is the identity, and hence has degree $0$. In particular, the degrees of $\nu^\dagger$ and $\mu^\dagger$ sum to zero. Now
\[ \mu^\dagger \in \Ext(\Od, \Oxx) \]
and since $\Od$ and $\Oxx$ are complexes concentrated in degree $0$, we see that $\mu^\dagger$ has non-negative degree. Similarly, we find
\[ \nu^\dagger \in \Ext(H^*(\Ox), k) \]
and since $k$ is concentrated in degree $0$ and all the terms in $H^*(\Ox)$ are in non-negative degrees we again find that the degree of $\nu^\dagger$ is non-negative. Since the degrees of $\mu^\dagger$ and $\nu^\dagger$ sum to $0$, they must both be zero and we can apply the previous result.

\printbibliography

\end{document}